\documentclass{amsart}
\usepackage[utf8]{inputenc}
\usepackage[english]{babel}
\usepackage{amsmath}
\usepackage{amsfonts}
\usepackage{amssymb}
\usepackage{amsthm}
\usepackage{hyperref}
\usepackage{mathtools}
\usepackage{longtable}
\usepackage{graphicx}
\usepackage{bbm}
\usepackage[backrefs,lite]{amsrefs}

\DeclareMathOperator{\im}{im}
\DeclareMathOperator{\Cl}{Cl}
\DeclareMathOperator{\Pic}{Pic}
\DeclareMathOperator{\lcm}{lcm}

\newcommand{\Z}{\mathbb{Z}}
\newcommand{\Q}{\mathbb{Q}}
\newcommand{\K}{\mathbb{K}}

\newtheorem{lem}{Lemma}[section]
\newtheorem{proposition}[lem]{Proposition}
\newtheorem{corollary}[lem]{Corollary}
\newtheorem{theorem}[lem]{Theorem}

\theoremstyle{definition}
\newtheorem{definition}[lem]{Definition}
\newtheorem{procedure}[lem]{Procedure}
\newtheorem{algorithm}[lem]{Algorithm}
\newtheorem{remark}[lem]{Remark}

\title[Calabi-Yau complete intersections in fwps]{Calabi-Yau complete intersections in \\ fake weighted projective spaces}
\author{Marco Ghirlanda}
\email{marco.ghirlanda@uni-tuebingen.de}
\address{Mathematisches Institut, Auf d. Morgenstelle 10, 72076 T\"ubingen}
\date{}
\begin{document}

\begin{abstract}
We present a classification algorithm for Calabi-Yau complete intersections arising from nef-partitions in fake weighted projective spaces, allowing us to determine all such complete intersections up to dimension five. Furthermore, we compute the Hodge pairs of the $3$-dimensional families obtained, and find twenty new Hodge pairs not realized by any toric Calabi-Yau hypersurface. Finally, we provide an explicit characterization for the families of maximal codimension.
\end{abstract}

\maketitle

\section{Introduction}

The study of Calabi-Yau hypersurfaces in the toric setting originates in Batyrev's work \cite{B}. With each reflexive polytope one can naturally associate a Gorenstein toric Fano variety hosting a family of Calabi-Yau hypersurfaces. Batyrev's construction identifies dual reflexive polytopes with mirror pairs of Calabi-Yau hypersurfaces. This viewpoint led to large-scale classifications of reflexive polytopes: Kreuzer and Skarke provided the complete list of reflexive polytopes in dimensions three \cite{KS3} and four \cite{KS4}. More recently, we developed in \cite{Ghi} a classification algorithm for reflexive simplices and explicitly determined them up to dimension six. The resulting dataset is available on Zenodo \cite{zenodo:17296449}.

A natural next step is to pass from hypersurfaces to complete intersections. In the toric framework,
Borisov introduced the notion of \emph{nef-partition}, namely a decomposition of the anticanonical class
into nef Cartier summands \cite{Bor}. Such a decomposition determines line bundles on the ambient toric
Fano variety. Intersecting general sections of each line bundle produces a family of Calabi-Yau complete intersections.

In this note we provide a classification algorithm for complete intersections arising from nef-partitions
in \emph{fake weighted projective spaces} (fwps), i.e.\ $\Q$-factorial toric Fano varieties of Picard number one.
Such varieties admit an explicit combinatorial description in terms of a \emph{degree matrix}, see
Section~\ref{sec:preliminaries}. In these terms, a nef-partition is encoded by a partition of the columns
into blocks, with the property that the column-sum of each block represents a Cartier divisor class.
We obtain in particular the following classification, available on Zenodo \cite{zenodo:CYCI-fwps}.

\begin{theorem}\label{big theorem}
The numbers of Calabi-Yau complete intersections of dimension up to five arising from nef-partitions in fake weighted projective spaces,
listed by their dimension $d$ and codimension $s$, are as follows:
\newpage
\begin{center}
\renewcommand{\arraystretch}{1.3}
\begin{longtable}{c|rrrrr}
{$s\backslash d$} &
\multicolumn{1}{c}{1} &
\multicolumn{1}{c}{2} &
\multicolumn{1}{c}{3} &
\multicolumn{1}{c}{4} &
\multicolumn{1}{c}{5} \\
\hline
\endfirsthead

{$s\backslash d$} &
\multicolumn{1}{c}{1} &
\multicolumn{1}{c}{2} &
\multicolumn{1}{c}{3} &
\multicolumn{1}{c}{4} &
\multicolumn{1}{c}{5} \\
\hline
\endhead

1 & 5 & 48 & 1.561 & 220.794 & 309.019.970 \\
2 & 2 & 10 & 164 & 6.045 & 1.042.424 \\
3 & - & 3 & 21 & 425 & 20.647 \\
4 & - & - & 6 & 43 & 1.134 \\
5 & - & - & - & 9 & 95 \\
6 & - & - & - & - & 18 \\
\end{longtable}
\end{center}
\end{theorem}
We then use PALP (see \cites{KS,BKSSW}) to compute the Hodge numbers of the $3$-dimensional families. Recall that in dimension three, thanks to its many symmetries, the Hodge diamond is determined entirely by
the Hodge pair $(h^{1,1},h^{2,1})$. The computations in dimension $4$ are still running.

\begin{corollary}\label{cor:new-hodge}
There are respectively $716$, $121$, $19$, and $6$ Hodge pairs realized by $3$-dimensional Calabi-Yau complete intersections
of codimension $1$, $2$, $3$, and $4$ arising from nef-partitions in fwps:
\begin{center}\label{fig:hodge_pairs}
\centering

\begin{minipage}[t]{0.4\textwidth}
  \centering
  \includegraphics[height=5cm,trim=90 0 100 0,clip]{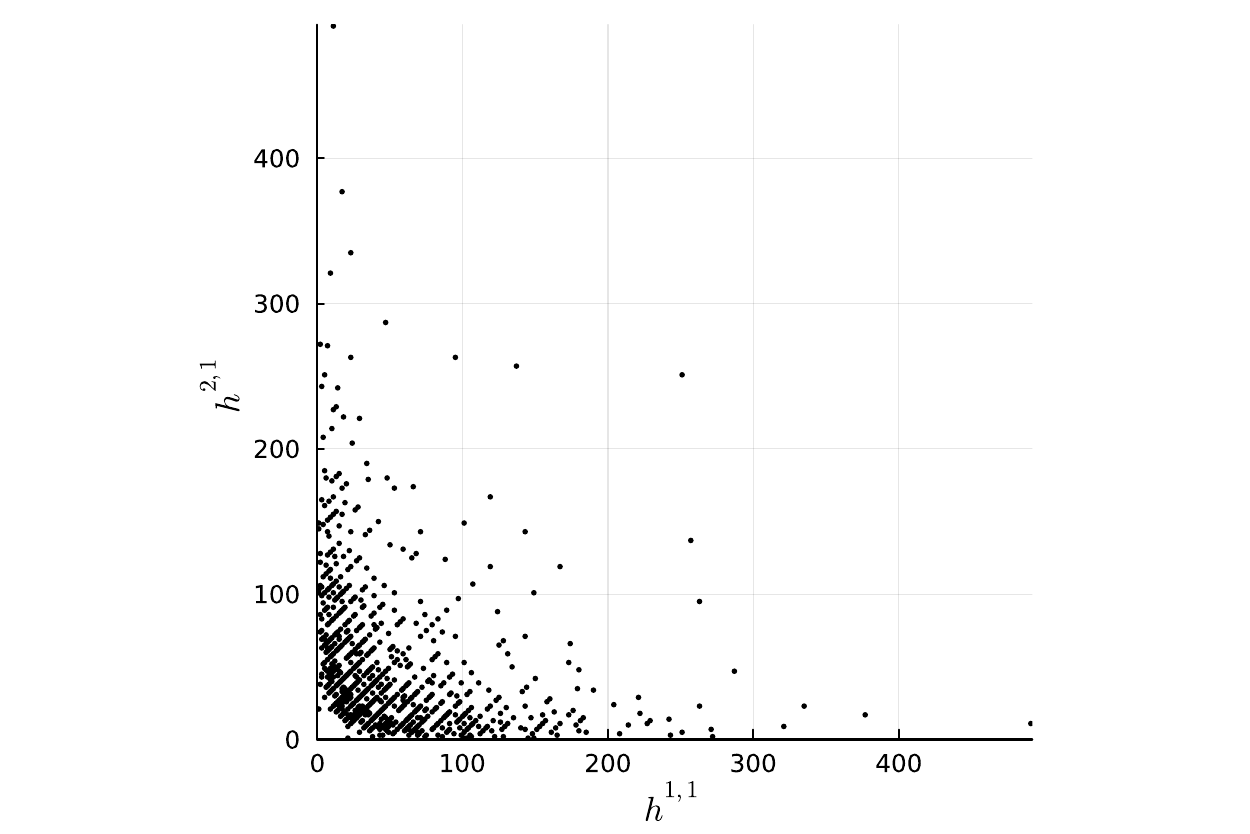}
\end{minipage}\hspace{-0.5cm}
\begin{minipage}[t]{0.4\textwidth}
  \centering
  \includegraphics[height=5cm,trim=190 0 170 0,clip]{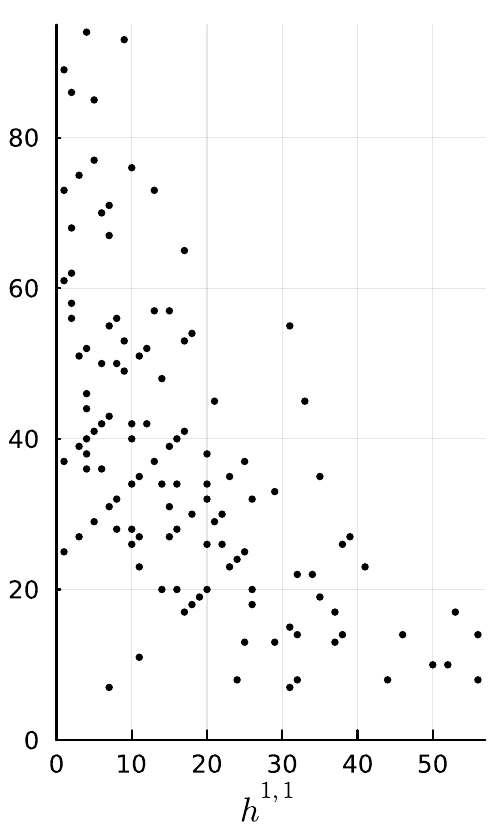}
\end{minipage}\hspace{-0.6cm}
\begin{minipage}[t]{0.13\textwidth}
  \centering
   \includegraphics[height=5cm,trim=260 0 235 0,clip]{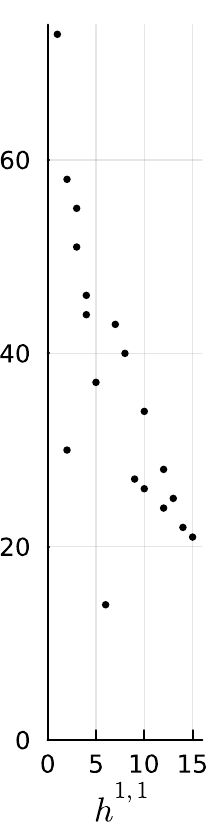}
\end{minipage}
\begin{minipage}[t]{0.13\textwidth}
  \centering
   \includegraphics[height=5cm,trim=280 0 250 0,clip]{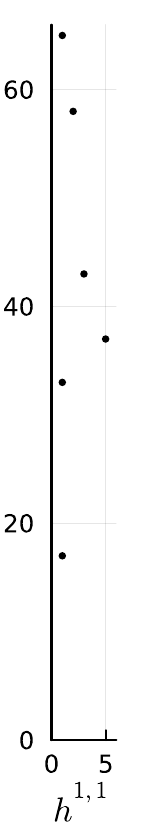}
\end{minipage}
\end{center}
Among the Hodge pairs arising in codimension $\ge 2$, exactly $20$ pairs are new, that means not realized by $3$-dimensional Calabi-Yau hypersurfaces
associated with $4$-dimensional reflexive polytopes (see \cite{KS4}).
For each such pair $(h^{1,1},h^{2,1})$ we indicate in subscript the codimensions in which the pair is realized:
\[
\begin{aligned}
&(1,25)_{2},\ (1,33)_{4},\ (1,37)_{2},\ (1,61)_{2},\ (1,65)_{4},\ (1,73)_{2,3},\ (1,77)_{4},\\[2pt]
&(1,89)_{2},\ (2,30)_{3},\ (2,56)_{2},\ (2,58)_{2,3,4},\ (2,68)_{2},\ (3,27)_{2},\ (3,39)_{2},\\[2pt]
&(3,55)_{3},\ (4,38)_{2},\ (6,14)_{3},\ (7,7)_{2},\ (11,11)_{2}.
\end{aligned}
\]
\end{corollary}

Our classification algorithm relies on a new combinatorial formulation of the existence condition for nef-partitions, which separates the constraints coming from the free and torsion parts of the class group, see Proposition \ref{prop:nef-part};
in particular, the conditions on the torsion coordinates can be verified row-wise. This yields a two-stage enumeration:
first list the admissible free weight data together with all nef-partitions, see Procedure \ref{proc:weights}, and then adjoin
torsion rows, retaining only those partitions that remain nef. The resulting procedure, presented in
Algorithm~\ref{algo:classification}, extends the algorithm of~\cite{Ghi} for reflexive simplices
to the nef-partition setting and produces the complete list of Calabi-Yau complete intersections of given dimension and codimension arising from nef-partitions in fwps.

Finally, in Section \ref{sec:maximal codimension} we provide an explicit characterization of the Calabi-Yau complete intersections of maximal codimension, showing that they are in bijection with the orbits of the multisets of points in a vector space over $\Z/2\Z$ affinely spanning the whole space, under the action of the affine general linear group.

\section{Preliminaries}\label{sec:preliminaries}

We first recall from \cite{HHHS}*{Sec. 2} the basic concepts on \emph{fake weighted projective space (fwps)}, that means $\Q$-factorial toric Fano varieties of Picard number one. The fan of a fwps of dimension $n$ is spanned by the faces of an $n$-dimensional Fano simplex~$\Delta$ with vertices $v_0,\dots,v_n\in\Z^{n}$. The fwps $Z$ is then encoded by the~$n\times (n+1)$ integral matrix
$$
P=\begin{bmatrix}
  v_0 & \dots & v_n  
\end{bmatrix}.
$$
The divisor class group of $Z$ is isomorphic to~$\Z^{n+1}/\im(P^*)$, where $P^*$ denotes the transpose of $P$. We present $\Cl(Z)$ in \emph{invariant factor form}:
$$\Cl(Z)=\Z\oplus \Z/\mu_1\Z\oplus\dots\oplus\Z/\mu_r\Z,$$
that means $\mu_r \mid \mu_{r-1} \mid \dots \mid \mu_1$.
Consider the projection $Q\colon \Z^{n+1}\rightarrow\Cl(Z)$, and let~$\omega_i=(w_i,\eta_i)\coloneqq Q(e_i)$, where $1\leq w_0\leq\dots\leq w_n$ and $\eta_i\in\Z/\mu_1\Z\oplus\dots\oplus\Z/\mu_r\Z$. We view $Q$ as a \emph{degree matrix in}~$\Cl(Z)$, that means
$$
Q=\begin{bmatrix}
    \omega_0 & \dots &\omega_n
\end{bmatrix}
$$
where any $n$ of the $\omega_i$ generate $\Cl(Z)$ as a group.
One can directly gain $Z$ as a quotient of $\K^{n+1}$ by the diagonal action of $H=\K^*\times F$ with weights $\omega_i$, where $F$ is finite and $H$ is the quasitorus with character group $\Cl(Z)$.

In \cite{Ghi}, we obtained the following explicit description of the Picard group of a fwps in terms of its degree matrix:

\begin{proposition} \label{prop:pic}
    Let $Z$ be a fwps with $\Cl(Z)=\Z\oplus \Z/\mu_1\Z\oplus\dots\oplus\Z/\mu_r\Z$ and degree matrix $Q=\begin{bmatrix}
        \omega_0 & \dots & \omega_n
    \end{bmatrix}$, where $\omega_i=(w_i,\eta_i)$.
    Let $0\leq\eta_{ij}'<\mu_j$ be the integer representing $\eta_{ij}$ and set
    \[
    \begin{array}{lcll}  
    L & \coloneqq & \lcm(w_0,\dots,w_n),
    \\[6pt]
    M_i & \coloneqq & \dfrac{\mu_i}{\gcd\left(\mu_i,\dfrac{L}{w_0}\eta_{0j}',\dots,\dfrac{L}{w_n}\eta_{nj}'\right)}, \quad 1 \le i \le r,
    \\[24pt]
    M & \coloneqq & \lcm(M_1,\dots,M_r).
    \end{array}
    \] 
    Then the Picard group of $Z$ is the subgroup of $\Cl(Z)$ generated by the element~$(LM,0)$.
\end{proposition}

In \cite{Bor}, Borisov introduced nef-partitions to construct families of Calabi-Yau complete intersections in toric varieties. We reformulate his definition in terms of degree matrices for the simplex case.

\begin{definition}\label{def:nef-part}
    Let $Z$ be a fwps with degree matrix $Q=[\omega_0,\dots,\omega_n]$. A \emph{nef-partition} of $Z$ is a partition $p_1\sqcup\dots\sqcup p_s=\{0,\dots,n\}$ such that
    $$\omega_{p_j}\coloneq\sum_{i\in p_j}\omega_i\in\Pic(Z), \hspace{5pt} j=1,\dots,s.$$
    The \emph{Calabi-Yau complete intersection (CYCI)} arising from such a partition is the family of complete intersections of general sections of $\omega_{p_1},\dots,\omega_{p_s}$ in $Z$.
\end{definition}

\begin{remark}\label{rem:linear-elimination} 
If $\omega_{p_j}=\omega_i$ for some indices $i$ and $j$, then a general section of~$\omega_{p_j}$ contains a monomial supported on the single
variable corresponding to $\omega_i$.
One can then eliminate this variable and the corresponding equation, obtaining an
isomorphic Calabi--Yau complete intersection in a fwps of dimension $d-1$ and codimension
$s-1$. Therefore, going forward we will restrict Definition \ref{def:nef-part} to nef-partitions for which
$\omega_{p_j}\neq \omega_i$ for all $j$ and all $i$, excluding in particular any partition having a
block of size one.
\end{remark}

\section{Classification of CYCI in fwps}

By Proposition \ref{prop:pic} and Definition \ref{def:nef-part} we have immediately the following characterization of nef-partitions in fwps:

\begin{proposition}\label{prop:nef-part}
    Let $Z$ be a fwps with $\Cl(Z)=\Z\oplus \Z/\mu_1\Z\oplus\dots\oplus\Z/\mu_r\Z$ and degree matrix $Q=\begin{bmatrix}
        \omega_0 & \dots & \omega_n
    \end{bmatrix}$. Then $p_1\sqcup\dots\sqcup p_s=\{0,\dots,n\}$ is a nef-partition if and only if, with the notation of Proposition \ref{prop:pic}, we have
    $$
    \begin{cases}
        L & | \ \ \sum_{i\in p_j} w_i,\quad j=1,\dots,s,\\[3pt]
        M_k & | \ \ \dfrac{\sum_{i\in p_j}w_i}{L},\quad j=1,\dots,s,\ k=1,\dots,r,\\[5pt]
        \sum_{i\in p_j}\eta_{ik} & =\ \ 0\in\Z/\mu_k\Z,\quad j=1,\dots,s,\ k=1,\dots,r.
    \end{cases}
    $$
\end{proposition}

As for \cite{Ghi}*{Proposition 4.1}, the conditions of Proposition \ref{prop:nef-part} can be checked independently on each torsion row. Accordingly, we adapt the method of \cite[Section 4]{Ghi} to our setting, generalizing the relevant definitions and procedures.

\begin{definition}\label{def: weight_vec}
    We call $w=[w_0,\dots,w_n]\in\Z^{n+1}_{\geq 1}$ an \emph{$s$-nef weight vector} if $w$ is the degree matrix of a fwps $Z$ with $\Cl(Z)=\Z$ and $Z$ admits a nef-partition $p_1\sqcup\dots\sqcup p_s=\{0,\dots,n\}$. We write $p(s,w)$ for the set of all such nef-partitions of $Z$, and call the pair $(w,p(s,w))$ an \emph{$s$-nef weight pair}.
\end{definition}

\begin{procedure}\label{proc:weights}
\emph{Input:} positive integers $n$ and $s$.\\ \emph{Procedure:}
\begin{enumerate}
  \item \label{weights:1} For $k=2,\dots,\left\lfloor\dfrac{n+1}{2}\right\rfloor$, compute the finite set $\mathcal{W}_k$ of all primitive non-decreasing $k$-tuples $(w_1,\dots,w_k)\in\Z_{\geq 1}^k$ satisfying $\lcm_i(w_i)\mid \sum_{i}w_i$, see for instance \cite{B}*{Algorithm 5.6};
  \item \label{weights:2} Compute the set $\mathcal{D}$ of all non-increasing $s$-tuples $(d_1,\dots,d_s)\in\Z_{\geq 2}^s$ satisfying $n+1=\sum_id_i$;
  \item \label{weights:3} For every $(d_1,\dots,d_s)\in\mathcal{D}$, form the set $\mathcal{W}$ of all $s$-tuples $(W_1,\dots,W_s)$, where $W_i\in\mathcal{W}_{d_i}$ and for every $i$ either $d_i>d_{i+1}$ or $W_i\geq_{lex}W_{i+1}$;
  \item \label{weights:4} For every $(W_1,\dots,W_s)\in\mathcal{W}$, denote with $L_i$, $S_i$ and $m_i$ respectively the least common multiple, the sum and the maximum of the coordinates of $W_i$. Then compute the set $\mathcal{C}(W_1,\dots,W_s)$ of all primitive $s$-tuples~$(c_1,\dots,c_s)\in\Z_{\geq1}^s$ such that:
  \vspace{3pt}
  \begin{enumerate}
      \item $c_1\mid \lcm(S_2,\dots,S_s)$;
      \vspace{3pt}
      \item $c_i L_i\mid c_jS_j$ for all $i,j=1,\dots,s$;
      \item $c_iS_i>c_jm_j$ for all $i,j=1,\dots,s$.
  \end{enumerate}
  \vspace{3pt}
  \item\label{weights:5} For every $(W_1,\dots,W_s)\in\mathcal{W}$ and $(c_1,\dots,c_s)\in\mathcal{C}(W_1,\dots,W_s)$, let $w$ be the concatenation of~$c_1W_1,\dots,c_sW_s$, reordered non-decreasingly. Then store~$w$ together with the partition $p_1\sqcup\dots\sqcup p_s=\{0,\dots,n\}$, where $p_i$ is the subset of indices corresponding to the elements of $c_iW_i$;
  \item \label{weights:6} Collect all partitions sharing the same $w$ in a set $q(s,w)$, then return all pairs $(w,q(s,w))$.
\end{enumerate}
\emph{Output}: the set of all $s$-nef weight pairs $(w,p(s,w))$ with $w\in\Z^{n+1}_{\geq 1}$.
\end{procedure}

\begin{proof}
Let $w=(w_0,\dots,w_n)\in\Z_{\ge 1}^{n+1}$ be an $s$-nef weight vector. By definition $w$ is primitive, and there exists a nef-partition
$p_1\sqcup\cdots\sqcup p_s=\{0,\dots,n\}$.
Reorder the parts so that $|p_1|\ge\cdots\ge |p_s|$, and set $d_i:=|p_i|$. By Remark \ref{rem:linear-elimination} we have $d_i\geq 2$ for all $i$, and hence $d_1\leq\frac{n+1}{2}$.
For each $i$, let $\widetilde W_i$ be the nondecreasing list of the weights $\{w_j : j\in p_i\}$ and define
\[
c_i:=\gcd(\widetilde W_i),\qquad W_i:=\widetilde W_i/c_i.
\]
Then $W_i$ is nondecreasing and primitive. Denote by
\[
L:=\lcm(w_0,\dots,w_n),\qquad L_i:=\lcm(W_i),\qquad S_i:=\sum_{v\in W_i} v.
\]
Since $\lcm(c_iW_i)=c_i\lcm(W_i)=c_iL_i$, we have $c_iL_i \mid L$ for all $i$.
Moreover, Proposition~\ref{prop:nef-part} gives
\[
L \ \mid \ \sum_{j\in p_i} w_j \ =\ c_i S_i \qquad \text{for each } i.
\]
It follows that $c_iL_i \mid c_iS_i$, hence $L_i \mid S_i$. Therefore $W_i\in\mathcal{W}_{d_i}$ for every $i$.
The tuple $(d_1,\dots,d_s)$ is an element of $\mathcal{D}$, and after possibly reordering blocks of equal length we may assume the lexicographical ordering of Step~\ref{weights:3}. Next, for all $i,j$ we have $c_iL_i \mid L$ and $L\mid c_jS_j$, hence condition (b) of Step~\ref{weights:4}. Condition (c) follows directly by Remark     \ref{rem:linear-elimination}.
Furthermore, since $w$ and $W_i$ are primitive, $(c_1,\dots,c_s)$ is also primitive.
Finally, let $p^m$ be a prime power dividing~$c_1$.
By primitivity of $(c_1,\dots,c_s)$ there exists $j\ge 2$ such that $p^m\nmid c_j$.
Since~$c_1L_1 \mid c_jS_j$, we have $p^m \mid c_jS_j$, hence~$p^m \mid S_j$.
Thus $c_1 \mid \lcm(S_2,\dots,S_s)$, and we conclude that $(c_1,\dots,c_s)\in\mathcal{C}(W_1,\dots,W_s)$.
Consequently, the procedure enumerates the data $(d_1,\dots,d_s)$, $(W_1,\dots,W_s)$ and $(c_1,\dots,c_s)$, and Step \ref{weights:5} reconstructs $w$ and its nef-partitions.

Conversely, consider any output pair $(w,q(s,w))$ coming from blocks $c_iW_i$. Then
\[
L \ =\ \lcm\bigl(\,\lcm(c_1W_1),\dots,\lcm(c_sW_s)\,\bigr)
\ =\ \lcm(c_1L_1,\dots,c_sL_s).
\]
By condition (b) of Step~\ref{weights:4} we have $c_iL_i \mid c_jS_j$ for all $i,j$, hence $L \mid c_jS_j$ for each $j$. Thus each partition in $q(s,w)$ is a nef-partition and $w$ is an $s$-nef weight vector.
\end{proof}

\begin{definition}\label{def: torsion_vec}
    Let $w$ be an $s$-nef weight vector and $\mu\in\Z_{\geq 2}$. We call~$\eta=[\eta_0,\dots,\eta_n]\in(\Z/\mu\Z)^{n+1}$ an \emph{$s$-nef torsion vector of order $\mu$ for $w$ }if the matrix
    $$Q=\begin{bmatrix}
        w_0&\dots & w_n\\
        \eta_0 & \dots & \eta_n
    \end{bmatrix}$$
    is the degree matrix of a fwps $Z=Z(Q)$ with $\Cl(Z)=\Z\oplus\Z/\mu\Z$ and $Z$ admits a nef-partition $p_1\sqcup\dots\sqcup p_s=\{0,\dots,n\}$. We write $p(s,\eta)\subseteq p(s,w)$ for the set of all such nef-partitions. and we call the pair $(\eta,p(s,\eta))$ an \emph{$s$-nef torsion pair of order~$\mu$ for $w$}.
    \end{definition}
    
    \begin{definition}
    Let $w$ be an $s$-nef weight vector and $\mu\in\Z_{\geq 2}$. We say that two $s$-nef torsion vectors $\eta$ and $\zeta$ of order $\mu$ for $w$ are \emph{equivalent} if there exists an automorphism of~$\Z\oplus\Z/\mu\Z$ sending $(w_i,\eta_i)$ to $(w_i,\zeta_i)$ for $i=0,\dots,n$. We write~$\eta\leq\zeta$ if $[\eta'_0,\dots,\eta'_n]\leq_{lex}[\zeta_0',\dots,\zeta_n']$, where $\eta_i'$ and $\zeta_i'$ are the representative integers for $\eta_i$ and $\zeta_i$ between $0$ and $\mu-1$. We say that $\eta$ is \emph{minimal} if $\eta\leq\zeta$ for all $\zeta$ equivalent to $\eta$.
\end{definition}

The following characterization of fwps admitting a nef-partition follows from Proposition \ref{prop:nef-part} and \cite{Ghi}*{Proposition 4.5}.

\begin{proposition}\label{prop:gor_matrix}
    Let $Z$ be a fwps with $\Cl(Z)=\Z\oplus\Z/\mu_1\Z\oplus\dots\oplus\Z/\mu_r\Z$. Then $Z$ admits a nef-partition $p_1\sqcup\dots\sqcup p_s=\{0,\dots,n\}$ if and only if $Z=Z(Q)$ for a degree matrix $$Q=\begin{bmatrix}
        w_0 & \dots & w_n\\
        \eta_{01} & \dots & \eta_{n1} \\
        \vdots & & \vdots \\
        \eta_{0r} & \dots & \eta_{nr} \\
    \end{bmatrix},$$
    where \begin{itemize}
        \item $w=[w_0,\dots,w_n]$ is an $s$-nef weight vector;
        \item $\eta_{\ast i}\coloneqq[\eta_{0i},\dots,\eta_{ni}]$ is a minimal $s$-torsion vector of order $\mu_i$ for $w$;
        \item $\eta_{\ast i}< \eta_{\ast j}$ for all $i<j$ such that $\mu_i=\mu_j$;
        \item the nef-partition lies in the intersection $\bigcap_{j=1}^rp(s,\eta_{\ast j})$.
    \end{itemize}.
\end{proposition}

\begin{proof}
    By Proposition \ref{prop:nef-part}, if $Z=Z(Q)$ for such a degree matrix then it admits a nef-partition $p_1\sqcup\dots\sqcup p_s=\{0,\dots,n\}$.
    Now assume $Z$ admits such a nef-partition, and let $Q$ be any degree matrix for $Z$. By Proposition \ref{prop:pic}, the first row $w$ of $Q$ is an $s$-nef weight vector, and the $(i+1)$-th row of $Q$ is an $s$-nef torsion vector of order $\mu_i$ for $w$, for $i=1,\dots,r$. By \cite[Theorem 2.1]{Ghi}, there exists an automorphism of $\Cl(Z)$ that turns all torsion vectors minimal and orders increasingly those with the same torsion order.
\end{proof}

\begin{algorithm}[CYCI in fwps]\label{algo:classification}
\emph{Input:} positive integers~$d$ and $s$. \\
\emph{Algorithm:}
\begin{enumerate}
  \item\label{step 1} Compute all $s$-nef weight pairs $(w,p(s,w))$ with $w\in\Z^{d+s+1}$ using Procedure \ref{proc:weights};
  \item\label{step 2} For each $s$-nef weight vector $w$ compute all pairs $(\mu,\eta)$ such that $\mu\in\Z_{\ge 2}$ and $\eta$ is a minimal $s$-nef torsion vector of order $\mu$ for $w$, using \cite{Ghi}*{Procedure 4.10} and filtering out those with $p(s,\eta)=\varnothing$;
  \item\label{step 3} Compute all degree matrices obtained by suitably combining an $s$-nef weight vector $w$ with its minimal $s$-nef torsion vectors from~(\ref{step 2}), using \cite{Ghi}*{Procedure 4.13} and filtering out at every step those with $\bigcap_{j}p(s,\eta_{\ast j})=\varnothing$;
  \item\label{step 4} Select a subset of the degree matrices from~(\ref{step 3}) that contains exactly one degree matrix for each isomorphism class of fwps, using \cite{Ghi}*{Procedure 4.16}. For each representative degree matrix $Q$, store it together with the set of its nef-partitions $\bigcap_{j=1}^r p(s,\eta_{\ast j})$;
  \item\label{step 5} For each stored degree matrix $Q$ and each nef-partition compute the multidegree of the associated complete intersection by summing the columns of $Q$ over each block, and identify nef-partitions yielding the same multidegree.
\end{enumerate}
\emph{Output:} the complete list of $d$-dimensional CYCI of codimension $s$ arising from nef-partitions in fwps.
\end{algorithm}

\begin{proof}
    Each step of the algorithm terminates by Procedure \ref{proc:weights} and the relevant Procedures from \cite{Ghi}*{Sec. 4} respectively. The output is correct and complete by Proposition \ref{prop:gor_matrix}.    
\end{proof}

\section{CYCI of maximal codimension}\label{sec:maximal codimension}

Complete intersections of maximal codimension exhibit many symmetries, which significantly slow down the isomorphism filtering in Step \ref{step 4} of Algorithm \ref{algo:classification}.  However, we can use these symmetries to provide an explicit characterization that renders their classification more tractable.

\begin{proposition}\label{prop:maximal codimension}
    Let $Z$ be a fwps of dimension $n$ with $\Cl(Z)=\Z\oplus \Z/\mu_1\Z\oplus\dots\oplus\Z/\mu_r\Z$ admitting a nef-partition $p_1\sqcup\dots\sqcup p_s=\{0,\dots,n\}$. Then $2s\leq n+1$. Furthermore, a fwps $Z$ of dimension $2s-1$ admits such a nef partition if and only if $\mu_1=\dots=\mu_r=2$ and $Z$ admits a degree matrix of the form
    $$
    Q = \begin{bmatrix}
    \mathbbm{1}_{s} & \mathbbm{1}_s \\
    A & A
    \end{bmatrix},\quad \mathbbm{1}_s=[1,\dots,1],
    $$
    where $A$ is an $r\times s$ matrix in $\Z/2\Z$ whose columns affinely span the whole $(\Z/2\Z)^r$.
\end{proposition}

\begin{proof}
    First, notice that, if $2s>n+1$, at least one of the blocks of the nef-partition would have size one, contradicting Remark \ref{rem:linear-elimination}.
    
    Now assume that $n=2s-1$. Then each block of the nef-partition has size two.
    Let $\omega_i=(w_i,\eta_i)$ and $\omega_j=(w_j,\eta_j)$ be two weights forming a block.
    By Proposition \ref{prop:nef-part}, we have $L \mid w_i + w_j$.
    By definition we have $w_i$, $w_j \le L$. Therefore, either~$w_i = w_j = L$ or $w_i+w_j=L$. In the latter case we have $w_i \mid L-w_i=w_j$ and viceversa, implying $w_i=w_j=\frac{L}{2}$.
    If all pairs $(w_i,w_j)$ were of the type $\left(\frac{L}{2},\frac{L}{2}\right)$, the lowest common multiple of the weights would also be $\frac{L}{2}$, contradicting the definition of $L$. Hence, at least one pair is of the type $(L,L)$. Also, since $\gcd(w_0,\dots,w_s)=1$, we have that either $L=1$ or $L=2$ and at least a pair has the form $(1,1)$. In the latter case, the sum of the pair would coincide with one of the other $w_i$, contradicting Remark \ref{rem:linear-elimination}. Hence, we conclude that $w_i=1$ for all $i$.

    By Proposition \ref{prop:nef-part}, it follows that $\mu_i=2$ for all $i$, and each block of the nef-partition is of the form $(\omega_i,\omega_i)$ with $\omega_i=(1,\eta_i)$. After appropriately permuting the columns of the degree matrix, we can therefore reach the form
    $$
    Q = \begin{bmatrix}
    \mathbbm{1}_{s} & \mathbbm{1}_s \\
    A & A
    \end{bmatrix}.
    $$
    We conclude by noticing that the almost freeness of $Q$ is equivalent to the condition that the columns of $A$ span the whole $(\Z/2\Z)^r$.
\end{proof}

\begin{proposition}\label{prop:multisets}
    Let $P(d)$ be the set of all multisets of $d+1$ points (counted with multiplicity) in $(\Z/2\Z)^r$ that affinely span the whole space.
    Then, the $d$-dimensional $CYCI$ arising from a nef partition of codimension $d+1$ in a fwps are in bijection with the orbits of $P(d)$ under the action of the group $\operatorname{AGL}(r, \Z/2\Z)$, described via the affine transformation:
    \[
    \alpha: \operatorname{AGL}(r, \Z/2\Z) \times (\Z/2\Z)^r \to (\Z/2\Z)^r, \quad \big((v, M), x\big) \mapsto Mx + v.
    \]
\end{proposition}

\begin{proof}
    Let $Z$ be the $(2d+1)$-dimensional ambient fwps for such a CYCI.
    By Proposition \ref{prop:maximal codimension}, $Z$ admits a degree matrix
    $$
    Q = \begin{bmatrix}
    \mathbbm{1}_{s} & \mathbbm{1}_s \\
    A & A
    \end{bmatrix},
    $$
    where $A$ is an $r\times s$ matrix in $\Z/2\Z$ whose column span is the whole $(\Z/2\Z)^r$. The columns of $A$ then form an element $P_A$ of $P(d)$, and viceversa every element of~$P(d)$ yield such a matrix $A$ by listing the points of the multiset as columns (with multiplicity).
    
    By \cite{Ghi}*{Theorem 2.1}, the degree matrix $Q$ yields isomorphic toric varieties up to performing elementary row operations on $A$, permuting its columns and adding $1$ to all entries of a row. In the correspondence with $P_A$, this is equivalent respectively to applying elements of $\text{GL}(r,\Z/2\Z)$ to the points of $P_A$, ignoring their order (automatic, since $P_A$ is a multiset) and traslating them by a vector $w\in(\Z/2\Z)^r$.
\end{proof}

\section*{Acknowledgments}
The author would like to thank Professor Victor Batyrev and Professor J\"urgen Hausen for their helpful suggestions and support during the preparation of this paper.

\end{document}